\documentclass[runningheads]{llncs}

\usepackage{graphicx, amsmath, amsfonts, mathtools, here}
\usepackage[english]{babel}
\usepackage[square,numbers]{natbib}
\usepackage{bm}
\usepackage{color}
\newtheorem{thm}{Theorem}

\makeatletter
    
    \@addtoreset{equation}{section}
\makeatother

\newcommand{\bsb}{\boldsymbol{b}}
\newcommand{\bsc}{\boldsymbol{c}}
\newcommand{\bsx}{\boldsymbol{x}}
\newcommand{\bsy}{\boldsymbol{y}}
\newcommand{\bss}{\boldsymbol{s}}

\newcommand{\bsA}{\boldsymbol{A}}

\newcommand{\bsepsilon}{\boldsymbol{\epsilon}}

\newcommand{\bsxi}{\boldsymbol{\xi}}

\newcommand{\bszeros}{\boldsymbol{0}}
\newcommand{\bsones}{\boldsymbol{1}}

\newcommand{\calC}{\mathcal{C}}
\newcommand{\calU}{\mathcal{U}}

\newcommand{\bbR}{\mathbb{R}}

\DeclareMathOperator{\LO}{LO}
\DeclareMathOperator{\RLO}{RLO}
\DeclareMathOperator{\IO}{IO}
\DeclareMathOperator{\st}{\text{s.t.}}

\usepackage{mathtools}
\mathtoolsset{showonlyrefs = true}

\begin{document}
\title{Inverse-Optimization-Based Uncertainty Set for Robust Linear Optimization}
%
%
\author{Ayaka Ueta\inst{1}\and Mirai Tanaka\inst{2,3} \and Ken Kobayashi\inst{1} \and Kazuhide Nakata\inst{1}}
\authorrunning{A. Ueta \textit{et al.}}
%
\institute{School of Engineering, Tokyo Institute of Technology \and
Department of Statistical Inference and Mathematics, The Institute of Statistical Mathematics \email{} \url{} \and
Continuous Optimization Team, RIKEN Center for Advanced Intelligence Project \email{}}
\maketitle              
\begin{abstract}

We consider solving linear optimization (LO) problems with uncertain objective coefficients. For such problems, we often employ robust optimization (RO) approaches by introducing an uncertainty set for the unknown coefficients. Typical RO approaches require observations or prior knowledge of the unknown coefficient to define an appropriate uncertainty set. However, such information may not always be available in practice. In this study, we propose a novel uncertainty set for robust linear optimization (RLO) problems without prior knowledge of the unknown coefficients. Instead, we assume to have data of known constraint parameters and corresponding optimal solutions. Specifically, we derive an explicit form of the uncertainty set as a polytope by applying techniques of inverse optimization (IO). We prove that the RLO problem with the proposed uncertainty set can be equivalently reformulated as an LO problem. Numerical experiments show that the RO approach with the proposed uncertainty set outperforms classical IO in terms of performance stability.

\end{abstract}

\section{Introduction}
The linear optimization (LO) problem considered in this paper is the following: 
\begin{align*}
    \begin{split}
        \min_{\bsx} \hspace{2mm}&\bsc^\top\bsx\\
        \hspace{52mm}\st\hspace{2mm}&\bsA\bsx=\bsb,\hspace{28mm}(\LO(\bsA, \bsb, \bsc))\\
        &\bsx\ge \bszeros,
    \end{split}
\end{align*}
where $\bsA\in\mathbb R^{m\times n}, \bsb\in\mathbb R^m$, and $\bsc\in\mathbb R^n$. 
In this paper, we assume the objective coefficient $\bsc$ is unknown. 
Robust optimization (RO) is a common approach for such optimization problems. 
In robust linear optimization (RLO)~\cite{ellip}, we assume $\bsA, \bsb$ to be known and $\bsc$ to take arbitrary values in a given uncertainty set $\calU\subseteq\mathbb R^n$.
The goal of RLO is to minimize the objective value for the worst-case scenario by solving the following problem:
\begin{align*}
    \begin{split}
        \min_{\bsx} \hspace{2mm}&\max_{\bsc\in\calU}\hspace{2mm}\bsc^\top\bsx\\
        \hspace{50mm}\st\hspace{2mm}&\bsA\bsx=\bsb,\hspace{24mm}(\RLO(\bsA, \bsb, \calU))\\
        &\bsx\ge \bszeros.
    \end{split}
\end{align*}
The above RLO problem gives us a robust optimal solution. 

In RO approaches in practice, the choice of uncertainty set $\mathcal U$ is crucial, and conventional approaches often rely on the availability of observations or prior knowledge of the uncertain coefficient $\bsc$. 
For example, \citet{ellip} introduced ellipsoidal uncertainty sets based on the sample mean vector and the sample covariance matrix of observations of $\bsc$, whose robust counterpart would be equivalently reformulated to a second-order cone optimization problem. 
Also, when we have observations of $\bsc$, we often use a polyhedral uncertainty set $\calU$ defined as a reduced convex hull~\citep{RCH, RCH_} of these observations, whose robust counterpart is an LO problem.
However, in many real-world scenarios, neither observations nor prior knowledge of $\bsc$ is readily available. This makes it difficult to define an appropriate uncertainty set $\mathcal U$ with conventional RO approaches.
Recognizing this gap, we consider constructing $\mathcal U$ with a given set of optimal solutions.

Let us consider transportation problems as a motivating example.  
In practice, the actual transportation costs among facilities depend on not only the distances but also other factors such as road quality and the number of traffic lights \citep{Chen2021}.
Consequently, it is often difficult to observe such transportation costs directly. 
Instead, we might easily access historical data on actual transportation quantities.  
These quantities are determined as optimal solutions to transportation problems with specific demands and supplies, which are observable parameters. 
Motivated by such scenarios, we aim to design an uncertainty set by such observable parameters and optimal solutions. 

In this study, we propose an uncertainty set based on the observed parameters and optimal solutions with the concept of inverse optimization (IO).
We assume that $\bsA$ is known, and that we have set of observed values $\{(\bsb_{k},\bsx_k^*)\}_{k \in \mathcal K}$ where $\bsx_k^*$ is a corresponding optimal solution to $\LO(\bsA, \bsb_{k}, \bsc)$ and $\mathcal K$ is a finite index set. Under this situation, we define an uncertainty set $\calU$ such that each observed optimal solution $\bsx_k^\ast$ lies within the optimal solution set of  $\LO(\bsA, \bsb_{k}, \bsc)$. 
For technical reasons, we assume $\bsc \ge \bszeros$ throughout the following discussion. 
Our main contributions are summarized as follows:
\begin{itemize}
    \item[(i)] 
    We provide an explicit formulation of the IO-based uncertainty set for RLO with uncertain objective coefficients without any prior knowledge of the unknown coefficient itself.

    \item[(ii)]
    We demonstrated that the RLO approach with the proposed uncertainty set outperformed the classical IO approach in terms of performance stability when we had no prior knowledge of the uncertain coefficient.
\end{itemize}

\section{Inverse Optimization (IO)}
\label{sect: IO}
IO is an approach to estimate unknown parameters in an optimization problem using known parameters and optimal solutions~\cite{IO}. 
Precisely, IO estimates $\bsc$ from observed $\bsA$, and $\{(\bsb_{k}, \bsx_{k}^{\ast})\}_{k \in \mathcal K}$ exploiting the optimality conditions of $\LO(\bsA,\bsb_k,\bsc)$. The necessary and sufficient condition for $\bsx_k^*$ to be an optimal solution of $\LO(\bsA,\bsb_k,\bsc)$ is that there exist $\bsy_k\in\mathbb R^m$ and $\bss_k\in\mathbb R^n$ such that
\begin{align*}
    \bsA^\top\bsy_k + \bss_k &= \bsc,\\
    \bss_k&\ge \bszeros,\\
    \bsx_k^{*\top}\bss_k&=0.
\end{align*}
Using these conditions, let
\begin{alignat}{2}
    C_{k} &\coloneqq \{\bsA^{\top} \bsy_{k} + \bss_{k}\colon \bss_{k} \ge \bszeros, \bsx_{k}^{\ast \top} \bss_{k} = 0\},&\quad
    \calC &\coloneqq \bigcap_{k \in \mathcal K} C_{k}.
    \label{eq: def inv feas set}
\end{alignat}
Then, it is necessary that $\bsc \in \calC$, which is so-called the \emph{inverse-feasible set}.
Note that $\calC$ is nonempty when $\bsx_k^*$ precisely satisfies the optimality condition for all $k\in\mathcal K$.
To estimate the unknown coefficient, the classical IO aims to find $\bsc\in\calC$ that is close to a given reference point $\hat\bsc$.
The classical IO is formulated as the following optimization problem:
\begin{align*}
    \hspace{52mm}\begin{split}
    \min_{\bsc}\hspace{2mm}&\|\bsc-\hat\bsc\|\\
    \st\hspace{2mm}&\bsc\in\calC,
    \end{split} \hspace{35mm}(\IO(\hat \bsc, \calC))
\end{align*}
where $\|\cdot\|$ represents an appropriate norm.

\section{IO-based Uncertainty Set}
As we have already seen the unknown coefficient $\bsc\in \calC$, it is natural to set $\calU=\calC$. 
However, since 
$\calC$ is a cone, when we set $\calU=\calC$, the optimal value of the inner maximization problem in RLO, i.e., $\max_{\bsc\in\calC}\bsc^\top\bsx$, is either zero or infinity for any $\bsx$. 

To avoid this, by introducing the $\ell_1$ normalization to $\calC$,
we propose an uncertainty set $\calU=\calC\cap\Delta$, where
$\Delta \coloneqq \{\bsc\in\mathbb R^n \colon \bsones^\top\bsc=1, \bsc\ge\bszeros\}$ is the unit simplex
and $\bsones \in \bbR^n$ is the all-one vector.
By fixing $\|\bsc\|_1 = 1$,
we can avoid that $\max_{\bsc\in\calU}\bsc^\top\bsx$ becomes zero or infinity for any $\bsx$.
The above definition makes $\calU$ a convex polytope and thus has the advantage of being able to leverage the strong duality theorem in the reformulation, as discussed in the next section. Moreover, we can easily prove the following (the proof is omitted).
\begin{thm}
For two sets of observations $\mathcal D_{i} \coloneqq \{(\bsb_k, \bsx_k^\ast)\}_{k \in \mathcal K_i}\ (i = 1, 2)$, let $\mathcal C_i$ be the corresponding inverse feasible set defined in the same manner as in Eq.\ \eqref{eq: def inv feas set} with a common $\bm A$, and $\mathcal U_i=\mathcal C_i\cap \Delta\ (i=1,2)$. 
If $\mathcal D_1\subseteq \mathcal D_2$, then $\calU_1\supseteq\calU_2$.
\end{thm}
This implies that the uncertainty set reduces in size as the number of observations increases. Although this is a natural and desirable property for uncertainty sets, it is not necessarily satisfied in existing uncertainty sets. 

\section{LO Reformulation of the Robust Counterpart}
$\RLO(\bsA, \bsb, \calU)$ has a nested structure of minimizing the maximum value, resulting in a two-stage optimization problem. For such RLO problems, a common approach is to reformulate it into a one-stage optimization problem. 
\begin{thm}
    $\RLO(\bsA, \bsb, \calU)$ with uncertainty set $\calU = \calC\cap\Delta$ can be equivalently reformulated as the following LO problem:
    \begin{align}
    \begin{split}
    \min_{\bsx,\{\bsxi_k\},\zeta}\hspace{2mm}&\zeta\\
    \mathrm{s.t.}\hspace{4mm}&\zeta\bsones \ge \sum_{k\in \mathcal K}\bsxi_k+\bsx,\\
    &\bsA\bsxi_k=\bszeros \hspace{15mm}(k\in \mathcal K),\\
    &[\bsxi_k]_j\ge 0 \hspace{15mm}(k\in \mathcal K; j\in \mathcal N_k),\\
    &\bsA\bsx = \bsb,\\
    &\bsx\ge\bszeros,
    \end{split}
    \label{eq:theorem}
    \end{align}
where $\mathcal N_{k} \coloneqq \{j \in \{1, \dots, n\}\colon [\bsx_{k}^{\ast}]_{j} = 0\}$ and $[\bm v]_j$ indicates the $j$-th element of vector $\bm v$.
\end{thm}

\begin{proof}
     Since $\calU = \calC\cap\Delta$,
     $\max_{\bsc\in\calU} \bsc^\top\bsx$ is equal to  the optimal value of the following LO problem:
\begin{align}
    \begin{split}
    \max_{\bsc,\{\bsy_k\},\{\bss_k\}}\hspace{2mm}&\bsc^\top\bsx\\
    \st\hspace{6.5mm}&\bsA^\top\bsy_k+\bss_k = \bsc \hspace{5mm}(k\in \mathcal K),\\
    &[\bss_k]_j\ge 0 \hspace{15mm}(k\in \mathcal K; j\in \mathcal N_k),\\
    &[\bss_k]_j = 0 \hspace{15mm}(k\in \mathcal K; j\in \mathcal B_k),\\
    &\bsones^\top\bsc = 1,\\
    &\bsc\ge\bszeros,
    \end{split}
    \label{eq:rLO}
\end{align}
where $\mathcal B_{k} \coloneqq \{j \in \{1, \dots, n\}\colon [\bsx_{k}^{\ast}]_{j} > 0\}$. 
The dual problem of the LO problem~\eqref{eq:rLO} is
\begin{align}
    \begin{split}
    \min_{\{\bsxi_k\},\zeta}\hspace{2mm}&\zeta\\
    \st\hspace{3mm}&\zeta\bsones \ge \sum_{k\in \mathcal K}\bsxi_k+\bsx,\\
    &\bsA\bsxi_k=\bszeros \hspace{15mm}(k\in \mathcal K),\\
    &[\bsxi_k]_j\ge 0 \hspace{15mm}(k\in \mathcal K; j\in \mathcal N_k).
    \end{split}
    \label{eq:rLO_}
\end{align}
As mentioned in Section \ref{sect: IO}, $\calU \ne \emptyset$,
that is, the primal problem \eqref{eq:rLO} and its dual problem \eqref{eq:rLO_} are feasible.
Hence from the strong duality theorem, their optimal values are equivalent. 
Therefore, $\RLO(\bsA, \bsb, \calU)$ with $\calU = \calC\cap\Delta$ is equivalent to the LO problem \eqref{eq:theorem}.\qed 
\end{proof}

\section{Numerical Experiments}
We evaluated the performances of our proposed uncertainty set with random instances. 
All optimization problems were solved using Python 3.7.13 and Gurobi 9.5.2 on a MacBook Pro 17.1 with Apple M1 CPUs and 16GB RAM. 

\paragraph{Data Generation}
We set $(m,n) = (10, 150)$ and randomly generated sets of observations $\{(\bsb_k, \bsx_{k}^\ast)\}_{k\in \mathcal K}$ as follows:  
The ground truth objective coefficient $\bsc\in \Delta$ was drawn from the Dirichlet distribution whose parameters were all one. 
We generated each element of $\bsA$ from a uniform distribution on the interval $[-1,1]$. 
Let $K$ be the number of observations.
For each $k\in \mathcal K\coloneqq\{1,\ldots,K\}$, we constructed a feasible solution $\bar \bsx_k$, whose elements were drawn from the uniform distribution on the interval $[0,1]$, and we set $\bsb_k = \bsA\bar \bsx_k$. 
Then, we obtained an optimal solution $\bsx_k^*$ by solving $\LO(\bsA,\bsb_k,\bsc)$. 
To evaluate the out-of-sample performance, we also constructed a validation set of  $\{(\bsb_l,\bsx_l^\ast)\}_{l\in \mathcal L}$ in the same way, where $L$ is the number of validation samples and  $\mathcal L\coloneqq\{K+1,\ldots, K+L\}$.

\paragraph{IO/RO Approaches and Evaluation Metric}
We compared the classical IO approach (Classical IO) and the RO approach with the proposed IO-based uncertainty set (IO-based RO).
In Classical IO, we first obtained $\bsc_{\text{IO}}$ by solving $\IO(\hat \bsc, \calC)$, where the $\ell_2$-norm was employed in the objective function. 
Then, we solved $\LO(\bsA,\bsb_l,\bsc_{\text{IO}})$ and obtained its optimal solution $\hat{\bsx}_l^{\text{IO}}$ for each $l\in \mathcal L$. 
Here, we compared two reference points:
$\hat{\bsc}_1\coloneqq(\boldsymbol{c}+\bsepsilon)/\|\boldsymbol{c}+\bsepsilon\|_1$ and $\hat{\bsc}_{2} \coloneqq (1 / n) \bsones$, where each entry of $\boldsymbol\epsilon$ was drawn from the uniform distribution on $[0, 0.001]$. 
The settings of $\hat{\bsc}= \hat{\bsc}_1$ and $\hat{\bsc}= \hat{\bsc}_2$ respectively simulate the cases where we do and do not have prior knowledge of $\bsc$.
In IO-based RO, 
 we obtained an optimal solution $\hat{\bsx}_l^{\text{RO}}$ to  $\text{RLO}(\bsA, \bsb_l, \calC\cap\Delta)$ by solving the problem~\eqref{eq:theorem} with $\bsb = \bsb_l$ for each $l\in \mathcal L$. 
To evaluate the out-of-sample quality of the obtained solutions $\hat{\bsx}_l\in  \{\hat{\bsx}_l^{\text{RO}}, \hat{\bsx}_l^{\text{IO}}\}$, we used the relative suboptimality gap defined by
\begin{equation}
    \ell(\hat{\bsx}_l, \bsx_l^{\ast})\coloneqq\frac{\bsc^\top\hat{\bsx}_l-\bsc^\top\bsx_l^{\ast}}{\bsc^\top\bsx_l^{\ast}},
\end{equation}
which indicates the relative difference between the objective values of $\hat{\bm x}_l$ with the ground truth objective coefficient and the optimal value. 

\paragraph{Results}
Fig.~\ref{fig:relloss} shows the boxplots of the relative suboptimality gap of the validation data set with $L=20$ for each $K\in \{10, 20, \dots, 130\}$. 
The line plots show the \textit{worst} relative suboptimality gap of each method. The results of Classical IO indicate that its performance was highly dependent on the choice of the reference point $\hat{\bsc}$. 
Comparing Classical IO with $\hat{\bsc}=\hat{\bsc}_2$ and our proposed IO-based RO, both the worst value and the variance of the relative suboptimality gap for IO-based RO were smaller than those of Classical IO, especially for small $K$. 
These results suggest that IO-based RO stably achieves good performance even when we do not have prior knowledge of the unknown coefficient and the number of observations is limited.
We also note that the computation of IO-based RO was finished within 0.1 seconds for all $K$. 

\begin{figure}[ht]
    \centering
    \includegraphics[width=100mm]{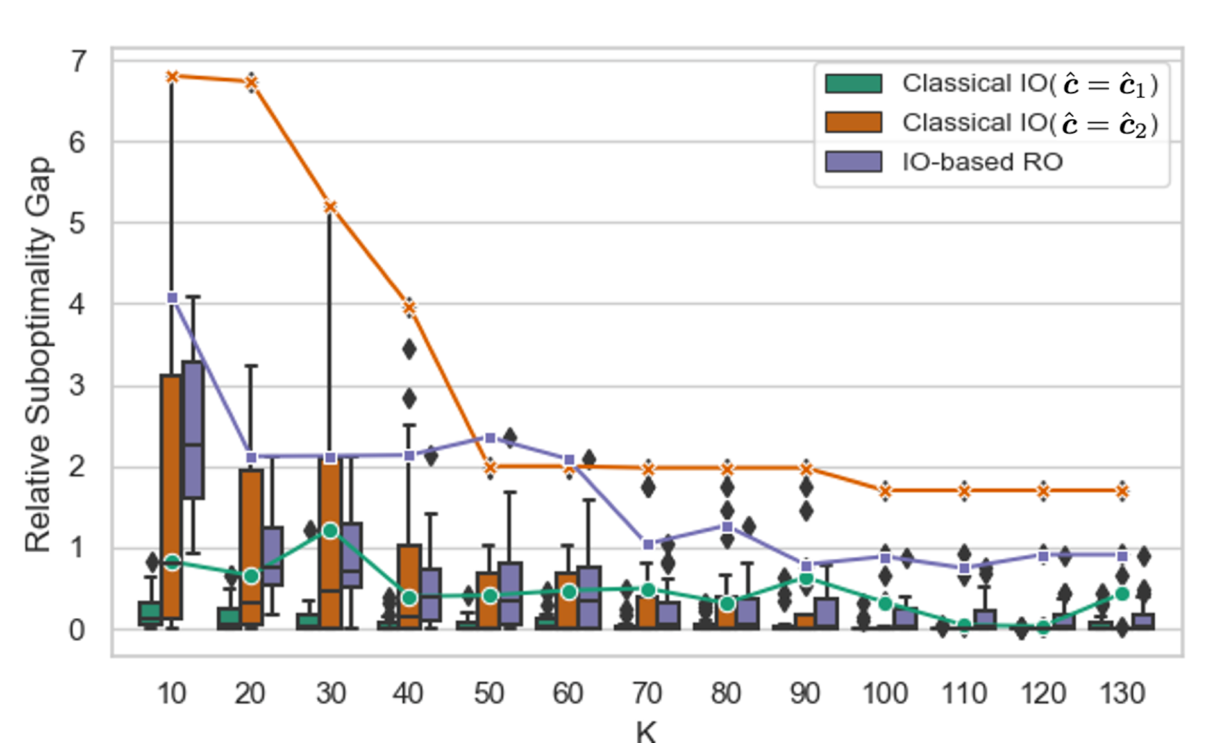}
    \caption{Boxplots of relative suboptimality gap with $L=20$}
    \label{fig:relloss}
\end{figure}

\section{Conclusion and Future Work}
In this paper, we proposed an IO-based uncertainty set for RLO problems with uncertain objective coefficients. We derived an explicit formulation of the uncertainty set without any prior knowledge of the unknown coefficient itself. We showed that the RO approach with our proposed uncertainty set has the advantage of performance stability when an appropriate reference point
is not available.

One of the limitations of this study is that our uncertainty set is derived from the duality theorem of  LO and thus is not available to optimization problems with discrete variables. 
Our future work includes extending this study to deal with mixed-integer optimization problems.

\bibliography{reference}
\end{document}